\documentclass[a4paper]{amsart}
\usepackage[all]{xy}
\usepackage[colorlinks=true]{hyperref}
\usepackage{enumerate}
\usepackage{amssymb}
\usepackage{comment}
\newtheorem{thm}{Theorem}[section]

\newtheorem{prop}[thm]{Proposition}
\newtheorem{lem}[thm]{Lemma}
\newtheorem{cor}[thm]{Corollary}

\theoremstyle{definition}
\newtheorem{defn}[thm]{Definition}
\newtheorem{ex}[thm]{Example}

\theoremstyle{remark}
\newtheorem{rem}[thm]{Remark}

\renewcommand{\d}{\mathrm d}                           

\newcommand{\set}[1]{\left\{#1\right\}}                
\newcommand{\eval}[1]{\left\langle#1\right\rangle}     
\newcommand{\br}{\left[\, , \, \right]}                
\newcommand{\brr}[1]{\left[#1\right]}                  
\newcommand{\cbrr}[1]{\left[\! \left[#1\right]\!\right]}

\newcommand{\NN}{\ensuremath{\mathcal{N}}}
\newcommand{\ds}{\displaystyle}                         

\newcommand{\C}{\ensuremath{\mathcal{C}}}  
\newcommand{\T}{{\mathcal{T}}}             
\renewcommand{\d}{\mathrm d}               
\newcommand{\Lie}{\boldsymbol{\pounds}}            

\newcommand{\smc}{\mbox{\,\tiny{$\circ $}\,}}         

\DeclareMathOperator{\graf}{graph}          

\newcommand{\al}{\alpha}
\newcommand{\be}{\beta}

\begin{document}

\title{On Poisson quasi-Nijenhuis Lie algebroids}
\author{Raquel Caseiro, Antonio De Nicola and Joana M. Nunes da Costa}
\address{CMUC, Department of Mathematics, University of Coimbra}
\email{raquel@mat.uc.pt, antondenicola@gmail.com, jmcosta@mat.uc.pt}

\date{June 13, 2008}

\begin{abstract}
We propose a definition of Poisson quasi-Nijenhuis Lie algebroids as
a natural generalization of Poisson quasi-Nijenhuis manifolds and
show that any such Lie algebroid has an associated quasi-Lie
bialgebroid. Therefore, also an associated Courant algebroid is
obtained. We introduce the notion of a morphism of quasi-Lie
bialgebroids and of the induced Courant algebroids morphism and
provide some examples of Courant algebroid morphisms. Finally, we
use paired operators to deform doubles of Lie and quasi-Lie
bialgebroids and find an application to generalized complex
geometry.

\end{abstract}
\maketitle

\section*{Introduction}             %
\label{sec:introduction}           %
The notion of Poisson quasi-Nijenhuis manifold was recently
introduced by Sti\'enon and Xu \cite{StiXu}. It is a manifold $M$
together with a Poisson bivector field $\pi$, a $(1,1)$-tensor $N$
compatible with $\pi$ and a closed $3$-form $\phi$ such that $i_N
\phi$ is also closed and the Nijenhuis torsion of $N$, which is
nonzero, is expressed by means of $\phi$ and $\pi$. When $\phi=0$
one obtains a Poisson-Nijenhuis manifold, a concept introduced by
Magri and Morosi \cite{MagriMorosi} to study integrable systems and
which was extended to the Lie algebroid framework by
Kosmann-Schwarzbach \cite{YKS} and Grabowski and Urbanski
\cite{GraUrb} who introduced the notion of a Poisson-Nijenhuis Lie
algebroid. In this paper we propose a definition of Poisson
quasi-Nijenhuis Lie algebroid, which is a straightforward
generalization of a Poisson quasi-Nijenhuis manifold.

Quasi-Lie bialgebroids were introduced by Roytenberg \cite{Roy} who
showed that they are the natural framework to study twisted Poisson
structures \cite{SevWei}. On the other hand, quasi-Lie bialgebroids
are intimately related to Courant algebroids \cite{LiuWeiXu},
because the double of a quasi-Lie bialgebroid carries a structure of
Courant algebroid and conversely, a Courant algebroid $E$ that
admits a Dirac subbundle $A$ and a transversal isotropic complement
$B$, can be identified with the Whitney sum $A \oplus A^*$, where
$A^*$ is identified with $B$ \cite{Roy}. Generalizing a result of
Kosmann-Schwarzbach \cite{YKS} for Poisson-Nijenhuis manifolds and
Lie bialgebroids, it is proved in \cite{StiXu} that a Poisson
quasi-Nijenhuis structure on a manifold $M$ is equivalent to a
quasi-Lie bialgebroid structure on $T^*M$. Extending the result of
\cite{StiXu}, we show that a Poisson quasi-Nijenhuis Lie algebroid
has an associated quasi-Lie bialgebroid, so that it has also an
associated Courant algebroid.

In an unpublished manuscript, Alekseev and Xu \cite{AleXu}, gave
the definition of a Courant algebroid morphism between $E_1$ and
$E_2$ and, in the case where $E_1$ and $E_2$ are doubles of Lie
bialgebroids $(A,A^*)$ and $(B,B^*)$, i.e $E_1=A \oplus A^*$ and
$E_2=B \oplus B^*$, they established a relationship with a Lie
bialgebroid morphism $A \to B$ \cite{mck2005}. Since doubles of
quasi-Lie bialgebroids are Courant algebroids, it seems natural to
obtain a relationship between Courant algebroid morphisms and
quasi-Lie bialgebroid morphisms. This is the case when considering
 Courant algebroids associated with a Poisson quasi-Nijenhuis Lie
algebroid of a certain type and with a twisted Poisson Lie
algebroid, respectively. In a first step towards our result, we
give the definition of a morphism of quasi-Lie bialgebroids which
is, up to our knowledge, a new concept that includes morphism of
Lie bialgebroids as a particular case.

Another aspect of Poisson quasi-Nijenhuis manifolds that is
exploited in \cite{StiXu} is the relation with generalized complex
structures. We extend to Poisson quasi-Nijenhuis Lie algebroids
some of the results obtained in \cite{StiXu} and also discuss the
relation of Poisson quasi-Nijenhuis Lie algebroids with paired
operators \cite{CarGraMar}.

The paper is divided into three sections. In section 1 we introduce
quasi-Lie bialgebroid morphisms and discuss their relationship with
Courant algebroid morphisms. Section 2 is devoted to Poisson
quasi-Nijenhuis Lie algebroids. We prove that each Poisson
quasi-Nijenhuis Lie algebroid has an associated quasi-Lie
bialgebroid and, in some particular cases, we construct a morphism
of Courant algebroids. In the last section we use paired operators
to deform doubles of Lie and quasi-Lie bialgebroids.

\section{Quasi-Lie bialgebroids morphisms}             %
\label{sec:Quasi-Lie bialgebroids}           %
\subsection{Quasi-Lie bialgebroids}

The main subject of this work are quasi-Lie bialgebroids.
 We begin by recalling the definition and give some examples.

\begin{defn}\cite{Roy}
A \emph{quasi-Lie bialgebroid} is a Lie algebroid $(A, \br_A,\rho)$
equipped with a degree-one derivation $\d_*$ of the Gerstenhaber
algebra  $\left(\Gamma(\wedge^\bullet A), \wedge, \br_A\right)$ and
a $3$-section of $A$, $X_A\in\Gamma(\wedge^{3} A)$ such that
\[\d_*X_A=0 \quad \mbox{and} \quad \d_*^2=\brr{X_A, - }_A.\]
\end{defn}

If $X_A$ is the null section, then $\d_*$ defines a structure of Lie
algebroid on $A^*$ such that $\d_*$ is a derivation of $\br_A$. In
this case we say that $(A,A^*)$ is a Lie bialgebroid.

Examples of quasi-Lie bialgebroids arise from different well known
geometric structures. We will illustrate some of them that will be
needed in our work.

\begin{ex}\label{ex:quasi:Lie:null:structure}
Let $(A, \br_A, \rho)$ be a Lie algebroid and consider any closed
3-form $\phi$. Equipping $A^*$ with the null Lie algebroid
structure,  $(A^*, \d_A, \phi)$ is canonically a quasi-Lie
bialgebroid.
\end{ex}

\begin{ex}\label{ex:twisted:Poisson:structure}[Lie algebroid with a twisted Poisson structure]
Let  $\pi\in\Gamma(\wedge^{2}A)$ be a bivector on the Lie
algebroid $(A,\br_A,\rho)$ and denote by $\pi^\sharp$ the usual
bundle map
\[
\begin{array}{llll}
  \pi^\sharp: &  A^\ast & \longrightarrow &A\\
  & \al &\longmapsto  & \pi^\sharp(\al)=i_\al \pi. \\
\end{array}
\]
This  map can be extended to a bundle map from
$\Gamma(\wedge^\bullet A^*)$ to $\Gamma(\wedge^\bullet A)$,  also
denoted by $\pi^\sharp$, as follows:
\[
\pi^\sharp(f)=f  \quad \mbox{and} \quad \eval{\pi^\sharp (\mu),
\al_1\wedge\ldots\wedge\al_k }=(-1)^k \mu \left(\pi^\sharp(\al_1),
\ldots, \pi^\sharp(\al_k)\right),
\]
for all $\ds f\in C^\infty(M)$ and $\ds \mu\in\Gamma(\wedge^k A^*)$
and  $\ds \al_1,\ldots, \al_k\in\Gamma(A^*)$.

Let $\phi\in \Gamma(\wedge^{3} A^*)$  be a closed 3-form on $A$. We
say that $(\pi, \phi)$ defines a \emph{twisted Poisson structure on
$A$} \cite{SevWei} if
 \[
 \brr{\pi,\pi}_A=2\,\pi^\sharp (\phi).
 \]
 In this case, the  bracket on the sections
of $A^\ast$  defined by
\begin{align*}
[\al,\be]^\phi_{\pi}&=\Lie_{\pi^\sharp\al}\be-\Lie_{\pi^\sharp\be}\al-\d\left(\pi(\al,\be)\right)+
\phi(\pi^\sharp\al, \pi^\sharp \be, - ), \quad \forall
\al,\be\in\Gamma(A^\ast),
\end{align*}
is a Lie bracket and $A^*_{\pi,\phi}=(A^\ast, \br_{\pi}^\phi,
\rho\circ\pi^\sharp)$ is a Lie algebroid. The differential of this
Lie algebroid is given by
\[\d^\phi_\pi X=[\pi , X]_A- \pi^\sharp( i_X\phi),\quad  \forall X\in \Gamma(A).\]

The pair $(A,A^\ast)$  is not a Lie bialgebroid but when we consider
the
  bracket on $\Gamma(A)$ defined by:
\[
\brr{X,Y}^\prime=\brr{X,Y}_A -\pi^\sharp( \phi(X,Y, - )), \quad
\forall X,Y\in\Gamma(A),
\]
the associated differential $\d^\prime$, given by
\begin{align*}\d^\prime f=\d f \quad \mbox{and} \quad  \d^\prime\al=\d\al - i_{\pi^\sharp\al}\phi, \quad  \forall f\in C^\infty(M), \al\in \Gamma(A^*), \end{align*}
defines on $A_{\pi,\phi}^*$ a structure of quasi-Lie bialgebroid
$(A_{\pi,\phi}^*, \d^\prime, \phi).$

One should notice that when $\phi=0$, $\pi$ is a Poisson bivector.
The Lie algebroid $A^*_{\pi,0}$ is simply denoted by $A^*_\pi$, and
together with the Lie algebroid $A$ it defines a special kind of Lie
bialgebroid called a \emph{triangular Lie bialgebroid}.
\end{ex}

 Any bundle map $\Phi:A\to B$ induces a map $\Phi^*: \Gamma(B^*)\to \Gamma(A^*)$ which assigns to each section $\al\in\Gamma(B^*)$
 the section
$\Phi^*\al$ given by
\[
\Phi^*\al(X)(m)=\eval{\al(\phi(m)), \Phi_m X(m)}, \quad \forall m\in
M, \, X\in \Gamma(A),
\]
where $\phi:M\to N$ is  the map induced by $\Phi$ on the base manifolds.
We denote by the same latter $\Phi^*$ the extension of this map to the multisections of $B^*$, where we set $\Phi^*f=f\smc \phi$, for $f\in C^\infty (N)$.

Let $A\to M$ and $B\to N$ be two Lie algebroids. Recall that a
\emph{Lie algebroid morphism} is a bundle map $\Phi:A\to B$
such that $\Phi^*:(\Gamma(\wedge^\bullet B^*), \d_B)\to (\Gamma(\wedge^\bullet A^*),\d_A)$ is a chain map.

Generalizing the notion of Lie bialgebroid morphism we propose the
following definition of morphism between quasi-Lie bialgebroids:

\begin{defn}\label{def:quasi:morphism}
Let $(A,\d_{A^*},X_A)$ and $(B, \d_{B^*}, X_B)$ be quasi-Lie
bialgebroids over  $M$ and $N$, respectively. A bundle map
$\Phi:A\to B$ is  a \emph{quasi-Lie bialgebroid morphism }if
\begin{itemize}
\item[1)] $\Phi$ is a Lie algebroid morphism;
\item[2)] $\Phi^*$ is compatible with the brackets on the sections of $A^*$ and $B^*$:
          \[\brr{\Phi^*\al, \Phi^*\be}_{A^*}=\Phi^*\brr{\al,\be}_{B^*};\]
\item[3)] the vector fields $\rho_{B^*}(\al)$ and $\rho_{A^*}(\Phi^*\al)$ are $\phi$-related:
           \[T\phi\cdot \rho_{A^*}(\Phi^*\al)=\rho_{B^*}(\al)\smc \phi;\]
\item[4)] $ \Phi X_A=X_B\smc\phi$,
\end{itemize}
where $\al$, $\be\in \Gamma(B^*)$ and $\phi:M\to N$ is the smooth
map induced by $\Phi$ on the base.
\end{defn}

\begin{ex}
A Lie bialgebroid morphism \cite{mck2005} is a Lie algebroid
morphism which is also a Poisson map, when we consider the
Lie-Poisson structures induced by their dual Lie algebroids. We can
easily see that in case we are dealing with Lie bialgebroids, the
definition of quasi-Lie bialgebroid morphism coincides with the one
of Lie bialgebroid morphism.
\end{ex}

\begin{ex} Consider $(A,\d_{A*},X_A)$ and $(B, \d_{B*}, X_B)$ two quasi-Lie bialgebroids
over the same base manifold $M$. We can see that a base preserving
quasi-Lie bialgebroid morphism (such that $\phi=\textrm{id}$) is a
bundle map $\Phi:A\to B$ such that $\Phi^*\smc \d_{B}=\d_{A}\smc
\Phi^*$, $\Phi\smc \d_{A^*}=\d_{B^*}\smc \Phi$ and $\ds \Phi
X_A=X_B$.\end{ex}

Other examples of quasi-Lie bialgebroid morphisms will appear in the
next section associated with quasi-Nijenhuis structures.

\subsection{Courant algebroids}
A \emph{Courant algebroid} $E\to M$ is a vector bundle over a
manifold $M$ equipped with a nondegenerate symmetric bilinear form
$\eval{ \, , \, }$, a vector bundle map $\rho: E\to TM$ and a
bilinear bracket $\smc$ on $\Gamma(E)$ satisfying:
 \begin{itemize}
 \item[C1)] $\ds e_1\smc(e_2\smc e_3)=(e_1\smc e_2)\smc e_3+e_2\smc(e_1\smc e_3)$
 \item[C2)] $\ds e\smc e=\rho^*  d\eval{e,e}$
 \item[C3)] $\ds \Lie_{\rho(e)}\eval{e_1,e_2}=\eval{e\smc e_1, e_2}+\eval{e_1, e\smc e_2}$
 \item[C4)] $\ds \rho(e_1\smc e_2)=\brr{\rho(e_1),\rho(e_2)}$
 \item[C5)] $\ds e_1\smc fe_2=f(e_1\smc e_2)+\Lie_{\rho(e_1)f}e_2$,
 \end{itemize}
for all $e,e_1,e_2,e_3\in\Gamma(E)$, $f\in C^\infty(M)$.

Associated with the bracket $\smc$, we can define  a skew-symmetric
bracket on the sections of $E$ by:
\[
\cbrr{e_1,e_2}=\frac{1}{2}\left(e_1\smc e_2- e_2\smc e_1\right)
\]
and the properties C1)-C5) can be expressed in terms of this
bracket.

\begin{ex}\label{{ex:Standard:Courant:algebroid}}[Standard Courant algebroid]
Let $(A,\br_A, \rho_A)$ be a Lie algebroid. The double $A\oplus A^*$
equipped with the skew-symmetric bracket
\[
\cbrr{X+\al,Y+\be}=\brr{X,Y}_A+ \left( \Lie_{X}\be - \Lie_Y\al
+\frac{1}{2}\d (\al(Y)-\be(X)) \right),
\]
the pairing $\eval{X+\al, Y+\be}=\al(Y) + \be(X)$ and the anchor
$\rho(X+\al)=\rho_A(X)$ is a Courant algebroid.
\end{ex}

A standard Courant algebroid is a simple example of a Courant
algebroid which is a the double of a Lie bialgebroid.
 The construction of Courant algebroids as doubles of Lie bialgebroids is implicit in the next example, where we explicit the construction of the double of a quasi-Lie bialgebroid.

\begin{ex}\label{double:quasi:Lie:bialgebroid} [Double of a
quasi-Lie bialgebroid] Let $(A, \d_*, X_A)$ be a quasi-Lie
bialgebroid. Its double $E=A\oplus A^*$ is a Courant algebroid if it
is equipped with the pairing $\eval{X+\al, Y+\be}=\al(Y) + \be(X)$,
the anchor $\rho=\rho_A+\rho_{A^*}$ and the bracket
\begin{align*}
\cbrr{X+\al, Y+\be}&=\brr{X,Y}_A+ \Lie^*_\al Y - \Lie^*_\be X -\frac{1}{2}\d_* (\al(Y)-\be(X))+ X_A(\al,\be, - )\\
&+ \left(\brr{\al,\be}_*+ \Lie_{X}\be - \Lie_Y\al + \frac{1}{2}\d
(\al(Y)-\be(X)) \right),
\end{align*}

Taking $X_A=0$ we have the Courant algebroid structure of a double
of a Lie bialgebroid.

\

Another particular case that worths to be mentioned is the double of
the quasi-Lie bialgebroid $(A^*, \d, \phi)$ illustrated
 in Example \ref{ex:quasi:Lie:null:structure}. In this case the anchor is simply  $\rho_E=\rho_A$ and the skew-symmetric bracket
  is a twisted version of the standard Courant bracket given by:
\begin{equation}\label{eq:exact:courant:algebroid}
\cbrr{X+\al,Y+\be}^\phi=\brr{X,Y}_A  + \Lie_{X}\be - \Lie_Y\al +
\frac{1}{2}\d (\al(Y)-\be(X))+  \phi(X,Y, - ).
\end{equation}
\end{ex}

\subsection{Dirac structures supported on a submanifold}
Dirac structures play an important role in the theory of Courant
algebroids. Let us recall them before proceed.

A \emph{Dirac structure on a Courant algebroid} $E$ is a subbundle
$A\subset E$, which is maximal isotropic with respect to the pairing
$\eval{\, , \,}$ and it is integrable in the sense that the space of
the sections of $A$ is closed under the bracket on $\Gamma(E)$.
Restricting the skew-symmetric bracket of $E$ and the anchor to $A$,
we endow the Dirac structure with a Lie algebroid structure $(A,
\cbrr{\, , \, }_{|_A}, \rho_{E_{|_A}})$. A Courant algebroid
together with a Dirac structure is called a \emph{Manin pair}.

As a way to generalize Dirac structures we have the concept of
generalized  Dirac structures or Dirac structures supported on a
submanifold of the base manifold.

\begin{defn}\cite{AleXu}
On a Courant algebroid $E\to M$, a \emph{Dirac structure supported
on a submanifold $P$ of $M$} or a \emph{generalized Dirac structure}
is a subbundle $F$ of $E_{|_P}$ such that:
 \begin{itemize}
 \item[D1)] for each $x\in P$, $F_x$ is maximal isotropic;
 \item[D2)] $F$ is compatible with the anchor, i.e.
         $\ds \rho_{|_P}(F)\subset TP $;
 \item[D3)]  For each $e_1, e_2\in \Gamma(E)$, such that ${e_1}_{|_P}$, ${e_2}_{|_P}\in \Gamma (F)$,
            we have $\ds ({e_1}\smc{e_2})_{|_P}\in \Gamma(F).$
 \end{itemize}
\end{defn}

Obviously, a Dirac structure supported on the whole base manifold
$M$ is  an usual   Dirac structure of the Courant algebroid.

Generalizing the Theorem $6.11$ on \cite{AleXu} to quasi-Lie
bialgebroids we have:

\begin{thm}\label{thm:Dirac:direct:sum}
Let $E=A\oplus A^*$ be the double of a quasi-Lie bialgebroid
$(A,\d_*, Q_A)$ over the  manifold $M$, $L\to P$ a vector subbundle
of $A$ over a submanifold $P$ of $M$ and $F=L\oplus L^\bot$. Then
$F$ is a Dirac structure supported on $P$ if and only if the
following conditions hold:
\begin{itemize}
 \item[1)] $L$ is a Lie subalgebroid of $A$;
 \item[2)] $L^\bot$ is closed for the bracket on $A^*$ defined by $\d_*$;
 \item[3)]  $L^\bot$ is compatible with the anchor, i.e., $\rho_{A^*|_P}(L^\bot)\subset TP$;
 \item[4)]  ${Q_A}_{|_{L^{\bot}}}=0$.
 \end{itemize}
\end{thm}

\begin{proof}
Since $F=L\oplus L^\bot$, this  is  a Lagrangian subbundle of $E$.
Suppose $F$ is a  Dirac structure supported on $P$. By definition,
we immediately deduce that $L$ is a Lie subalgebroid of $A$  and,
for $\al$, $\be$ sections of $A^*$ such that $\al_{|_P}$,
$\be_{|_P}\in \Gamma(L^\bot)$, we have
\[(\al\smc \be)_{|_P}=-Q_A(\al,\be, - )_{|_P}+\brr{\al,\be}_{A^*|_P}\in \Gamma(L\oplus L^\bot),\]
 and this  means that
$\brr{\al,\be}_{A^*|_P}\in L^\bot$ and $Q_A(\al,\be, - )_{|_P}\in
L$, or equivalently, $L^\bot$ is closed with respect to the bracket
of $E$ and ${Q_A}_{|_{L^{\bot}}}=0$.

Moreover, since $F$ is compatible with the anchor,
$$\rho_{A^*|_P}(\al_{|_P})=\rho_{A^*}(\al)_{|_P}=\rho_E(\al)_{|_P}\in TP,$$
so $L^\bot$ is compatible with $\rho_{A^*}$.

Conversely, suppose $L$ is a Lie subalgebroid of $A$, $L^\bot\subset
A^*$ is closed  for $\br_{A^*}$, ${\rho_{A^*}}_{|_P}(L^\bot)\subset
TP$ and $Q_{A|_{L^\bot}}=0$. Obviously $F$ is  compatible with the
anchor. We are left to prove that $F$ is closed with respect to the
bracket on $E$. Let $X, Y\in \Gamma(A)$ and $\al, \be\in
\Gamma(A^*)$ such that $X+\al$ and $Y+\be$ restricted to $P$ are
sections of $F$, then
\begin{align*}
{(X+\al)\smc (Y+\be)}_E & =\brr{X,Y}_A+i_{\al} \d_* Y -i_{\be} \d_* X +\d_*\left(\al(Y)\right)-Q_A(\al,\be, -) \\
& + \brr{\al,\be}_{A^*}+\Lie_{X} \be-i_Y \d \al.
\end{align*}

By hypothesis, we immediately have  that
$$\brr{X,Y}_{A|_P}=\brr{X_{|_P},Y_{|_P}}_L\in \Gamma(L), $$
$$\brr{\al,\be}_{A^*|_P}=\brr{\al_{|_P},\be_{|_P}}_{L^{\bot}}\in \Gamma(L^\bot) $$
\mbox{and} $$Q_A(\al,\be, -)_{|_P}\in   \Gamma(L).$$

Now, notice that $\al(Y)_{|_P}=0$, so $d\al(Y)_{|_P}\in
\nu^*(P)=(TP)^0$. Since $\rho_{A^*|_P}(L^\bot)\subset TP$, we  have
that
$$\d_{*} \al(Y)_{|_P}=\rho^*_{A^*|_P}d\al(Y)\in \Gamma(L).$$
Analogously, $\ds \d \al(Y)_{|_P}=\rho^*_{A}d\al(Y)_{|_P}\in
\Gamma(L^\bot).$

Also, $$\d_* Y(\al,\be)_{|_P}=\left(\rho_{A^*}(\al)\cdot \be(Y)-
\rho_{A^*}(\be)\cdot \al(Y)-\brr{\al,\be}_{A^*}(Y)\right)_{|_P}=0,$$
so $i_{\al}\d_* Y\in \Gamma(L)$. Analogously, $\ds i_{X}\d \be\in
\Gamma(L^{\bot})$.

All these conditions allow us to say that $(X+\al)\smc (Y+\be) \in
\Gamma(L\oplus L^\bot)$ and, consequently, $F$ is a Dirac structure
supported on $P$.
\end{proof}

\begin{cor}\cite{AleXu}
Let $E=A\oplus A^*$ be the double of a Lie bialgebroid then
$F=L\oplus L^\bot$ is a Dirac structure supported on $P$ if and only
if $L$ and $L^\bot$ are Lie subalgebroids of $A$ and $A^*$.
\end{cor}

Notice that when $P=M$ we obtain Proposition 7.1 of \cite{LiuWeiXu}.

\begin{cor}\cite{AleXu}
Let $E=TM\oplus T^*M$ be the standard Courant algebroid twisted by
 the 3-form $\phi\in \Omega^{3}(M)$ (see equation (\ref{eq:exact:courant:algebroid}) in Example \ref{double:quasi:Lie:bialgebroid}). For any submanifold $P$ of
$M$, $F=TP\oplus \nu^*P$ is a Dirac structure supported on $M$ iff
$i^*\phi=0$, where $i:P\hookrightarrow M$ is the inclusion map.
\end{cor}

Like Lie bialgebroids morphisms, quasi-Lie bialgebroid morphisms
give rise to Courant algebroid morphisms. Let us recall what is a
Courant algebroid morphism.

\begin{defn}\cite{AleXu}
A \emph{Courant algebroid morphism} between two Courant algebroids
$E\to M$ and $E'\to M'$ is a Dirac structure in $E\times \overline
E'$ supported on $\graf \phi$, where $\phi:M\to M'$ is a smooth map
and $\overline E'$ denotes the Courant algebroid obtained from $E'$
by changing the sign of the bilinear form.
\end{defn}

\begin{thm}\label{thm:quasi:Courant:morph}
Let $E_1=A\oplus A^*$ and $E_2=B\oplus B^*$ be doubles of quasi-Lie
bialgebroids $(A,\d_{A*}, Q_A)$ and $(B,\d_{B*}, Q_B)$ and
$(\Phi,\phi):A\to B$ a quasi-Lie bialgebroid morphism, then
\[
F=\set{(a+\Phi^*b^*,\Phi a+ b^*)| a\in A \mbox{ and } b^*\in B^*
\mbox{over compatible fibers}}\subset E_1\times \overline{E_2}
\]
is a Dirac structure supported on $\graf \phi$, i.e. $F$ is a
Courant algebroid morphism.
\end{thm}

\begin{proof}
The idea of the proof is analogous to the idea of the  proof of
Theorem 6.10 in \cite{AleXu} for Lie bialgebroid morphisms.

Consider $M$ and $N$ the base manifolds of $A$ and $B$,
respectively. Consider the following subbundles over $\graf \phi$
$$L=\graf \Phi=\set{(a,\Phi a)|\, a\in A}\subset A\times B$$
and
$$L^\bot=\set{(\Phi^*b^*, -b^*)|\, b^*\in B^*}\subset A^*\times {B^*}.$$

Since $\Phi$ is a Lie algebroid morphism, $L$  is clearly a Lie
subalgebroid of $A\times B$. Analogously, we can also conclude that
$L^\bot$ is closed for the bracket on  $A^*\times \overline{B^*}$
(where $\overline{B^*}$ denotes the bundle $B^*$ with  bracket
$\br_{\overline{B^*}}=-\br_{B^*}$) and it is compatible with the
anchor $\rho_{A^*\times \overline{B^*}}=(\rho_{A^*},-\rho_{B^*})$.
Also, since $\Phi Q_A=Q_B\smc \phi$, we have that
$(Q_A,Q_B)_{|L^\bot}=0$. So, Theorem \ref{thm:Dirac:direct:sum}
guarantees  that $L\oplus L^\bot$  is a Dirac structure  supported
on $\graf \phi$ of the double $A\times B\oplus A^*\times
\overline{B^*}$ which is the Courant algebroid $A\oplus A^*\times
B\oplus \overline{B^*}$. Finally, observe that the bundle morphism
$b+ b^*\mapsto b-b^*$ induces a canonical isomorphism between $F$
and $L\oplus L^\bot$ and the result follows.
\end{proof}

\section{Poisson Quasi-Nijenhuis Lie algebroids}             %
\label{sec:Quasi:Poisson:Nij}                            %

Let $(A,\br,\rho)$ be a Lie algebroid over a manifold $M$.  The
torsion  of a bundle map $N:A\to A$ (over the identity) is defined
by
\begin{equation}
\label{eq:Nijenhuis} \T_N(X,Y):=[NX,NY]-N[X,Y]_N, \quad  X,Y\in
\Gamma(A),
\end{equation}
where $\br_N$ is given by:
$$
[X,Y]_N:=[NX,Y]+[X,NY]-N[X,Y],\quad X,Y\in \Gamma(A).
$$

When $\T_N=0$, the bundle map $N$ is called a \emph{Nijenhuis
operator}, the triple $A_N=(A,\br_N,\rho_N=\rho\smc N)$ is a new Lie
algebroid  and  $N:A_N\to A$ is a Lie algebroid morphism.

\begin{defn}
On a Lie algebroid $A$ with a Poisson structure $\pi\in\Gamma(\wedge^2 A)$, we say that a
bundle map $N:A\to A$ is \emph{compatible} with $\pi$ if
$N\pi^\sharp=\pi^\sharp N^*$ and the \emph{Magri-Morosi concomitant} vanishes:
\[\C(\pi,N)(\al,\be)= \brr{\al,\be}_{N\pi}-\brr{\al,\be}_\pi^{N^\ast}=0,\]
where $\br_{N\pi}$ is the  bracket defined by the bivector field
$N\pi\in\Gamma(\wedge^{2} A)$, and $\br_\pi^{N^\ast}$ is the Lie
bracket obtained from the Lie bracket $\br_\pi$ by deformation along
the  tensor $N^\ast$.
\end{defn}

As a straightforward generalization of the definition of
quasi-Poisson Nijenhuis manifolds presented in \cite{StiXu}, we
have:

\begin{defn}\label{def:quasiNijLiealg}
A \emph{Poisson quasi-Nijenhuis Lie algebroid} $(A, \pi, N, \phi)$
is a Lie algebroid $A$ equipped with a Poisson structure $\pi$, a
bundle map $N:A\to A$ compatible with $\pi$ and a closed 3-form
$\phi\in\Gamma(\wedge^{3} A^*)$ such that
\[
\T_N(X,Y)=-\pi^\sharp\left(i_{X\wedge Y} \phi \right) \mbox{ and }
\d i_N\phi=0.
\]
\end{defn}

\begin{thm}\label{theor:quasi:Lie:bialg}
If $(A, \pi, N, \phi)$ is a Poisson quasi-Nijenhuis Lie algebroid
then $(A_\pi^*, \d_N, \phi)$ is a quasi-Lie bialgebroid.
\end{thm}

\begin{proof}
First notice that  $\d\phi=0$ and  $\d i_N\phi=0$ imply that
\[
\d_N\phi=\brr{i_N,\d}\phi=i_N\d\phi - \d i_N\phi=0.
\]

Secondly, we notice that since the bundle morphism $N$ and the Poisson structure
$\pi$ are compatible, then $\d_N$ is a derivation of the Lie bracket
$\br_\pi$. In fact, first one directly  sees that $\d$ is a derivation of $\br_{N\pi}$ and, since $\C(\pi,N)$ vanishes and
\begin{align*}
\d (\C(\pi,N)(\al,\be))&=\d_N\brr{\al,\be}_\pi-\brr{\d_N\al,\be}_\pi-\brr{\al,\d_N\be}_\pi\\
&-\d\brr{\al,\be}_{N\pi}+\brr{\d\al,\be}_{N\pi}+\brr{\al,\d\be}_{N\pi},
\end{align*}
we immediately conclude that $\d_N$ is a derivation of $\br_\pi$
(the particular case where $A=TM$ can be found in \cite{YKS}).

\

It remains to prove that $\d_N^{2}=\brr{\phi, - }_\pi$. Using the
definition of $\d_N$, we have:
\begin{align*}
\d_N^{2}\al(X,Y,Z)&= \T_N(X,Y)\eval{\al,Z} - \eval{\al,
\brr{\T_N(X,Y),Z}+\T_N(\brr{X,Y},Z)} + \mbox{c.p.}
\end{align*}
The fact that $\T_N(X,Y)=-\pi^\sharp i_{X\wedge
Y}\phi$ yields:
\begin{eqnarray*}
\lefteqn{\d_N^{2}\al(X,Y,Z)= -\phi(X,Y, \pi^\sharp\d\eval{\al,Z}) - \eval{\al, \Lie_{Z}\left( \pi^\sharp i_{X\wedge Y}\phi\right) - \pi^\sharp i_{\brr{X,Y}\wedge Z }\phi} + \mbox{c.p.}}\\
&=& -\phi(X,Y, \pi^\sharp\d\eval{\al,Z}) - \eval{\al, \left(\Lie_{Z} \pi\right)^\sharp i_{X\wedge Y}\phi + \pi^\sharp \left( \Lie_Z i_{X\wedge Y }\phi\right)} \\
&&+ \phi(\brr{X,Y}, Z, \pi^\sharp \al) + \mbox{c.p.}\\
&=&  -\phi(X,Y, \pi^\sharp\d\eval{\al,Z}) - \eval{\al, \left(\Lie_{Z} \pi\right)^\sharp i_{X\wedge Y}\phi + \pi^\sharp  \left(i_{X\wedge Y }\Lie_Z \phi\right) -\pi^\sharp i_{\brr{Z, X\wedge Y}}\phi}\\
&& + \phi(\brr{X,Y}, Z, \pi^\sharp \al) + \mbox{c.p.}\\
&=&-\phi(X,Y, \pi^\sharp\d\eval{\al,Z}) - \phi(X,Y, \left(\Lie_{Z} \pi\right)^\sharp \al) -\Lie_Z \phi(X,Y, \pi^\sharp\al)\\
&&- \phi(X, \brr{Z,Y}, \pi^\sharp \al)   + \mbox{c.p.}.
\end{eqnarray*}

Since
\begin{align*}
\brr{\phi, \al}_{\pi}(X,Y,Z)
&=-\Lie_{\pi^\sharp(\al)} \phi(X,Y,Z) \\
&- \set{\phi (X,Y,\pi^\sharp\d \eval{\al,Z}) - \phi
(X,Y,\Lie_Z\pi^\sharp(\al)) + \mbox{c.p.}},
\end{align*}
 and  by hypothesis,   $\phi$ is closed, we finally have that
\[
(\d_N^{2}\al - \brr{\phi , \al}_\pi)(X,Y,Z)= -\d\phi (X,Y,Z,
\pi^\sharp \al)=0.
\]
\end{proof}

Suppose $(A,\pi,N,\phi)$ is a Poisson quasi-Nijenhuis Lie algebroid.
The double of the quasi-Lie bialgebroid $(A_\pi^*, \d_N, \phi)$ is a
Courant algebroid (see Example \ref{double:quasi:Lie:bialgebroid})
that we denote by $E_\pi^\phi$.

An interesting case is when the $3$-form $\phi$ is  the image by
$N^*$ of another closed $3$-form $\psi$:
\[
\phi=N^*\psi \quad \mbox{and} \quad \d\psi=0.
\]
In this case $(A, N\pi, \psi )$ is a twisted Poisson Lie algebroid
because
\[\brr{N\pi,N\pi}=2\pi^\sharp(\phi)=2\pi^\sharp(N^*\psi)=2N\pi^\sharp(\psi)\]
and $A^*$  has a structure  of Lie algebroid:
$A^{*\psi}_{N\pi}=(A^*, \br_{N\pi}^\psi, N\pi^\sharp )$ (see
Example \ref{ex:twisted:Poisson:structure}). Equipping $A$ with the differential $\d^\prime$ given by
\[
\d^\prime f=\d f, \quad \mbox{and} \quad
\d^\prime\al=\d\al-i_{N\pi^\sharp \al}\psi,
\]
  for  $f\in C^\infty(M)$ and $\al\in\Gamma(A^*)$,
we obtain a quasi-Lie bialgebroid: $(A^{*\psi }_{N\pi}, \d^\prime,
\psi)$. Its double is a Courant algebroid and we denote it by
$E_{N\pi}^\psi$.

\begin{thm}
Let $(A, \pi, N, \phi)$ be a Poisson quasi-Nijenhuis Lie algebroid
and suppose that $\phi=N^*\psi$, for some closed $3$-form $\psi$,
then
 \[
 F=\set{\left(a+N^*\al, Na+\al \right)| a\in A \mbox{ and } \al \in A^* } \subset E_{N\pi}^\psi\times \overline{E_\pi^\phi}
 \]
defines a Courant algebroid morphism between $E_{N\pi}^\psi$ and
$E_\pi^\phi$.
\end{thm}

In order to prove the theorem, we need to remark the following
property.

\begin{lem}
Let $(A, \pi, N, \phi)$ be a Poisson quasi-Nijenhuis Lie algebroid,
then
\[
\eval{\T_{N^*}(\al,\be),X}=\phi(\pi^\sharp\al, \pi^\sharp\be, X),
\]
for all $X\in\Gamma(A)$ and  $\al,\be\in\Gamma(A^*)$.
\end{lem}

\begin{proof}
The compatibility between $N$ and $\pi$ implies that (see
\cite{KosMagri})
\begin{align*}
\eval{\T_{N^*}(\al,\be),X}&=\eval{\al, \T_N(X, \pi^\sharp\be)},
\end{align*}
so
\begin{align*}
\eval{\T_{N^*}(\al,\be),X}&=\eval{\al, -\pi^\sharp\left(i_{X\wedge
\pi^\sharp\be}\phi\right)}=-\phi(X,\pi^\sharp\be,\pi^\sharp
\al)=\phi(\pi^\sharp \al,\pi^\sharp \be,X).
\end{align*}
\end{proof}

\begin{proof}[Proof of the Theorem]
First notice that $N^*: A^{*\psi}_{N\pi}\to A^*_{\pi}$ is a Lie
algebroid morphism because it is obviously compatible with the
anchors and
\begin{align*}
N^*\brr{\al,\be}_{N\pi}^\psi&=N^*\brr{\al, \be}_{N\pi} + N^*\psi(\pi^\sharp \al, \pi^\sharp \be, - )\\
&= \brr{N^*\al, N^* \be}_\pi - \T_{N^*}(\al,\be) + \phi(\pi^\sharp
\al, \pi^\sharp \be, -)= \brr{N^*\al, N^* \be}_\pi.
\end{align*}

Let $\br^\prime$ be the bracket on the sections of $A$ induced by the
differential $\d^\prime$. Notice that
$$\brr{X,f}^\prime=\eval{\d^\prime f, X}= \eval{\d f, X}, $$
so $\ds N^*\d^\prime f=\d_N f$, for all $f\in C^\infty(M)$ and $X\in
\Gamma(A)$.

And since
\[
\brr{X,Y}^\prime= \brr{X,Y} - (N\pi)^\sharp(\psi(X,Y,  - )),
\]
we have:
\begin{align*}
N\brr{X,Y}_N&= \brr{NX, NY} - \T_N(X,Y)= \brr{NX,NY} + \pi^\sharp(i_{X\wedge Y}N^*\psi)\\
&= \brr{NX,NY} + \psi (NX, NY, N\pi^\sharp - )= \brr{NX,NY}^\prime,
\end{align*}
for all $X,Y\in \Gamma(A)$.

This way we conclude that  $N^*:A^{*\psi}_{N\pi}\to A^*_{\pi}$ is a
quasi-Lie bialgebroid morphism (see definition
\ref{def:quasi:morphism}) and the result follows from Theorem
\ref{thm:quasi:Courant:morph}.
\end{proof}

\begin{rem} As a trivial particular case, consider $A$ a Lie algebroid and $\psi$  a closed 3-form. We have that
 $(A^*, \d, \psi)$ is quasi-Lie bialgebroid (see Example \ref{ex:quasi:Lie:null:structure}). If $N:A\to A$ is a Nijenhuis operator, then $A_N=(A,\br_N, \rho\smc N)$ is a Lie algebroid and $N:A_N\to A$ is a  Lie algebroid morphism. So, $$\d_N N^*\psi=N^* \d\psi=0$$ and $(A_N, \pi=0, N, N^*\psi)$ is a quasi-Nijenhuis Lie algebroid if $\d N^*\psi=0$.  Then
 $(A^*, \d_N, \phi=N^*\psi)$ is a quasi-Lie bialgebroid and $N^*:(A^*,\d, \psi)\to (A^*,\d_N, N^*\psi)$ is obviously a quasi-Lie bialgebroid morphism.

 \

 In fact, we can directly and immediately see that $(A^*, \d_N, N^*\psi)$ is a quasi-Lie bialgebroid (without thinking about quasi-Nijenhuis structures). This way we avoid the condition $\d N^*\psi$ needed to prove that   $(A_N, \pi=0, N, N^*\psi)$ is a quasi-Nijenhuis algebroid and $N^*:(A^*,\d, \psi)\to (A^*,\d_N, N^*\psi)$ is obviously still a quasi-Lie bialgebroid morphism.
\end{rem}

\section{Paired operators}  %


Let $(A,\d_{A^*},X_A)$ be a quasi-Lie bialgebroid over $M$ and
consider a bundle map  over the identity,  ${\mathcal N}: A \oplus
A^* \to A \oplus A^*$. This bundle map can be written in the matrix
form ${\mathcal N}= \left (
\begin{array}{cc}
 N & \pi  \\
 \sigma & N_{A^*}
 \end{array}
 \right )
 $
 with $N:A \to A$, $N_{A^*}: A^* \to A^*$, $\pi:A^* \to A$
 and $\sigma: A \to A^*$.

\begin{defn}
The operator ${\mathcal N}$ is called \emph{paired} if $$\langle
X+\al,\, {\mathcal N}(Y+ \be) \rangle + \langle {\mathcal
N}(X+\al),\, Y+ \be \rangle=0,$$ for all $X+\al, Y+\be \in A\oplus
A^*$, where $\langle \cdot, \cdot \rangle$ is the usual pairing
 on the double $A \oplus A^*$.
\end{defn}

As it is observed in \cite{CarGraMar}, ${\mathcal N}$ is paired if
and only if $\pi \in \Gamma(\wedge^{2} A)$, $\sigma \in
\Gamma(\wedge^{2} A^*)$ and $N_{A^*}= -N^*$.

\subsection{Paired operators on the double of Lie bialgebroids}

Let us now take the Lie bialgebroid $(A,A^*)$, where $A^*$ has the
null Lie algebroid structure. In this case, the double
 $A \oplus A^*$ is the standard Courant algebroid of Example \ref{{ex:Standard:Courant:algebroid}}.

Now we consider, on the sections of $A \oplus A^*$, the bracket
deformed by ${\mathcal N}$,
\begin{equation*}
\cbrr{X+\al, Y+\be}_{{\mathcal N}}=\cbrr{{\mathcal N}(X+\al),
Y+\be}+ \cbrr{X+\al, {\mathcal N}(Y+\be)}- {\mathcal N}\cbrr{X+\al,
Y+\be}
\end{equation*}
and the Courant-Nijenhuis torsion of ${\mathcal N}$,
\begin{equation*}\label{torsion}
\T_{\NN}(X+\al,Y+\be):=\cbrr{{\mathcal N}(X+\al),{\mathcal
N}(Y+\be)}- {\mathcal N}\cbrr{X+\al, Y+\be}_{{\mathcal N}}.
\end{equation*}

A simple computation shows that for all $\al,\be \in \Gamma(A^*)$,
$$\cbrr{\al, \be}_{\mathcal N}=[\al, \be]_{\pi}.
$$

\begin{prop} \label{A*Lie:algeb}
Let ${\mathcal N}$ be a paired operator on $A \oplus A^*$. If
$\T_{{\mathcal N}\,|A^*}=0$, then the vector bundle $A^*$ is
equipped with the Lie algebroid structure $A^*_\pi$.
\end{prop}

\begin{proof}
A straightforward computation shows that $$\T_{\mathcal N}(\al,
\be)=0 \,\, \Rightarrow \,\,  [\pi ^\# \al, \pi ^\# \be]= \pi^\#
[\al,\be]_{\pi},$$ for all sections $\al$ and $\be $ of $A^*$. This
means that $\pi$ is a Poisson bivector on $A$ and the result
follows.
\end{proof}

Now we give sufficient conditions for a paired operator to define a
Poisson quasi-Nijenhuis structure on a Lie algebroid.

\begin{thm}\label{teo:Poisson:q-N:alg}
Let ${\mathcal N}= \left (
\begin{array}{cc}
 N & \pi  \\
 \sigma & -N^*
 \end{array}
 \right )
 $  be a paired operator on $A\oplus A^*$ such that
 \[N \pi^\# =\pi ^\# N^*\quad
\mbox{and} \quad  i_{NX} \sigma = N^*( i_X\sigma),\quad  \forall X
\in \Gamma(A).\] If $\T_{{\mathcal N}\,|A^*}=0$ and $\T_{{\mathcal
N}\,|A}=0$, then $(A,\pi,N, \d\sigma)$ is a Poisson quasi-Nijenhuis
Lie algebroid.
\end{thm}

\begin{proof}
First, notice that the condition $i_{NX} \sigma = N^*( i_X\sigma)$
means that
$$\sigma(NX,Y)=\sigma(X,NY), \quad \forall X,Y \in \Gamma(A),$$ and
implies that $N \sigma$ defined by $N \sigma(X,Y)= \sigma(NX,Y)$ is
a $2$-form on $M$. The condition $N \pi^\# =\pi ^\# N^*$ ensures
that $N \pi$ is a bivector field on $A$.

For all $\al,\be \in \Gamma(A^*)$,
$$\T_{\mathcal N}(\al,\be)=0 \quad \mbox{\em iff} \quad \pi \quad \mbox{is a Poisson bivector} \quad \mbox{and}
 \quad [\al, \be]_{\pi}^{N^*}=[\al, \be]_{N\pi}.$$
So we have that $\pi$ and $N$ are compatible. On the other hand, if
$X$ and $Y$ are sections of $A$, then
$$\T_{\mathcal N}(X,Y)=0 \quad \mbox{\em iff} \quad \T_N(X,Y)=
\pi^\#(\d \sigma(X, Y, -)) \quad \mbox{{\rm and}} \quad \d(N
\sigma)= i_N \d \sigma.$$ According to Definition
\ref{def:quasiNijLiealg}, $(A, \pi, N, \d \sigma)$ is a Poisson
quasi-Nijenhuis Lie algebroid.
\end{proof}
From Theorem \ref{theor:quasi:Lie:bialg}, we obtain:

\begin{cor} \label{cor:Poisson:q-N:alg}
$(A^*_{\pi},\d_N, \d \sigma)$ is a quasi-Lie bialgebroid.
\end{cor}

\begin{rem} \label{rem:N:square}
We note that a paired operator ${\mathcal N}$ that satisfies
${\mathcal N}^{2} = -\textrm{Id}_{A \oplus A^*}$, also satisfies
$$\langle {\mathcal N}(X+\al),\, {\mathcal N}(Y+ \be) \rangle =
\langle X+\al, Y+ \be \rangle, \quad \forall X+\al, Y+\be \in
\Gamma(A \oplus A^*).$$ In this case ${\mathcal N}$ defines a
\emph{generalized complex structure on the Lie algebroid} $A$.
From ${\mathcal N}^{2} = -\textrm{Id}_{A \oplus A^*}$ we deduce
that $N \pi^\# =\pi ^\# N^*$, $N^{2} X + \pi^\#(i_X \sigma)= -X$
and $i_{NX} \sigma = N^*( i_X\sigma),$ with $X \in \Gamma(A)$.
\end{rem}

Let us denote by $(A\oplus A^*)_{{\mathcal N}}$ the vector bundle
map equipped with the nondegenerate symmetric bilinear form $\langle
\cdot, \cdot \rangle_{{\mathcal N}}$ given by \[\langle X+\al, Y+\be
\rangle_{{\mathcal N}}=  \langle {\mathcal N}(X+\al), {\mathcal
N}(Y+\be) \rangle,\] the bundle map
 $\rho_{{\mathcal N}}$ given by $\rho_{{\mathcal N}}(X+\al)= a(NX)+\pi^\#(\al)$ and
the bracket $\cbrr{\,\, ,\, }_{{\mathcal N}}$  on its space of sections.

We can now establish a result that generalizes the one of
\cite{StiXu}, for the case where the Lie algebroid $A$ is $TM$.

\begin{thm} \label{deformed:double:quasiLie}
Let ${\mathcal N}= \left (
\begin{array}{cc}
 N & \pi  \\
 \sigma & -N^*
 \end{array}
 \right )
 $ be a paired operator on $A \oplus A^*$ such that ${\mathcal N}^{2} = -\textrm{Id}_{A \oplus
 A^*}$. If
$\T_{\NN\,|A^*}=0$ and $\T_{\NN\,|A}=0$, then $(A\oplus
A^*)_{{\mathcal N}}$ is a Courant algebroid and it is identified
with the double of the quasi-Lie bialgebroid $(A^*_{\pi}, \d_N, \d
\sigma)$.
\end{thm}

\begin{proof}
From Corollary \ref{cor:Poisson:q-N:alg} and Remark
\ref{rem:N:square} we have a Courant algebroid  $E^{\d
\sigma}_{\pi}$ which is the double of the quasi-Lie bialgebroid
$(A^*_{\pi},\d_N, \d \sigma)$. An easy computation shows that the
bracket on $\Gamma(E^{\d \sigma}_{\pi})$ coincides with the bracket
$\cbrr{\,\, , \, }_{{\mathcal N}}$ on $\Gamma((A\oplus A^*)_{{\mathcal
N}})$, the anchor of $E^{\d \sigma}_{\pi}$ is $\rho_{{\mathcal N}}$
and the nondegenerate bilinear form on $E^{\d \sigma}_{\pi}$ is
exactly $\langle \cdot, \cdot \rangle_{{\mathcal N}}$.

\end{proof}

\subsection{Paired operators on the double of quasi-Lie bialgebroids}
Now we consider the quasi-Lie bialgebroid $(A^*,\d_A, \phi)$ of
Example \ref{ex:quasi:Lie:null:structure} and the Courant algebroid
structure on its double: the standard Courant bracket twisted by
$\phi$, $\cbrr{\,\, ,\, }^\phi$,  and the anchor $\rho_A$. Let
${\mathcal N}= \left (
\begin{array}{cc}
 N & \pi  \\
 \sigma & -N^*
 \end{array}
 \right )
 $ be a paired operator and consider the bracket on
 $\Gamma(A \oplus A^*)$ deformed by ${\mathcal N}$:

\begin{equation*}
\cbrr{X+\al, Y+\be}^{\phi}_{{\mathcal N}}=\cbrr{{\mathcal N}(X+\al),
Y+\be}^{\phi}+ \cbrr{X+\al, {\mathcal N}(Y+\be)}^{\phi}- {\mathcal
N}\cbrr{X+\al, Y+\be}^{\phi}.
\end{equation*}

The Theorem \ref{deformed:double:quasiLie} admits a direct extension
 for the case of quasi-Lie bialgebroids.

\begin{thm}
Let ${\mathcal N}$ be a paired operator on the double $A \oplus A^*$
of the quasi-Lie bialgebroid $(A^*,\d_A, \phi)$, such that
${\mathcal N}^{2}=-Id_{A \oplus A^*}$. If $\T_{\NN\,|A^*}=0$ and
$\T_{\NN\,|A}=0$, then $(A\oplus A^*)^{\phi}_{{\mathcal N}}=(A\oplus
A^*, \cbrr{\,\, , \,}^{\phi}_{{\mathcal N}}, \rho_{{\mathcal N}}, \langle
\cdot , \cdot \rangle_{{\mathcal N}})$ is a Courant algebroid and it
is identified with the double of the quasi-Lie bialgebroid
$(A^*_{\pi}, \d', \d \sigma + i_N \phi)$, where $\d'$ the
differential given by $\d'f=\d_N f$ and $\d'\al= \d_N \al -
i_{\pi^\#(\al)} \phi$, for $f \in \C^\infty(M)$ and $\al \in
\Gamma(A^*)$.
\end{thm}

\begin{proof}
Let $\al, \be \in \Gamma(A^*)$ and $X,Y \in  \Gamma(A)$. Then,

\[
\T_{\mathcal N}(\al,\be)=0 \quad \mbox{\em iff} \quad \left\{
\begin{array}{ll}
 &\pi \quad \mbox{is a Poisson bivector}
 \\
\\ &[\al, \be]_{N\pi} - [\al, \be]_{\pi}^{N^*}=  \phi(\pi ^\#(\al),
\pi ^\#(\be), -) \end{array} \right.
 \]

 and
$\T_{\mathcal N}(X,Y)=0$ \emph{iff}

\[
\left\{
\begin{array}{ll} &\T_N(X,Y)=
\pi^\#((\d \sigma + i_N \phi)(X, Y, -))- N \pi^\#(\phi(X,Y,-))\\
\\

& \d(N \sigma)(X,Y,-) + \phi(X,Y,-)\\ &=
\phi(NX,NY,-)+\phi(NX,Y,N-)+ \phi(X,NY,N-) + (i_N \d \sigma)(X,Y,-).
\end{array}
\right.\]

 A straightforward
 generalization for Lie algebroids of the results presented in \cite{Vaisman} and in
 \cite{LinMTZab} in the case of a manifold, establishes that the four equations corresponding to
$\T_{\NN\,|A^*}=0$ and $\T_{\NN\,|A}=0$ are equivalent to the
vanishing of the Courant-Nijenhuis torsion of ${\mathcal N}$ with
respect to the bracket $\cbrr{\,\, ,\,  }^\phi$. Therefore, we have a
new Courant algebroid structure on the vector bundle $A \oplus A^*$,
$(A \oplus A^*)_{{\mathcal N}}^{\phi}=(A \oplus A^*, \cbrr{\,\,
,\,}^{\phi}_{{\mathcal N}}, \rho_{{\mathcal N}}, \langle \cdot ,
\cdot \rangle_{{\mathcal N}})$.

The restriction of the bracket $\cbrr{\, \, ,\,}^{\phi}_{{\mathcal N}}$
to the sections of $A^*$ is the bracket $[\, ,\, ]_{\pi}$ and since
$\T_{\NN\,|A^*}=0$, we have that $A^*_{\pi}$ is a Dirac structure of
the Courant algebroid $(A \oplus A^*)_{{\mathcal N}}^{\phi}$. On the
other hand, the restriction of the bracket $\cbrr{\,\,
,\,}^{\phi}_{{\mathcal N}}$ to the sections of $A$ gives
$$\cbrr{X ,Y}^{\phi}_{{\mathcal N}}= [X,Y]_N - \pi^\#( \phi(X,Y,-))
+ \d \sigma(X,Y,-) + i_N \phi(X,Y,-)$$  and the anchor
$\rho_{{\mathcal N}}$ restricted to $\Gamma(A)$  is $\rho_A \circ
N$. If we consider the bracket
$$[X,Y]'= [X,Y]_N - \pi^\#( \phi(X,Y,-))$$
on the sections of $A$ and the bundle map $\rho_A \circ N$, the
differential corresponding to this structure on $A$ is $\d'$ given
by,

$$\d'f =\d_N f \quad \mbox{and} \quad \d'\al= \d_N \al -
i_{\pi^\#(\al)} \phi,$$ with $f \in \C^\infty(M)$ and $\al \in
\Gamma(A^*)$.

The vector bundle $A$ is obviously a transversal isotropic
complement of $A^*$, so that $(A^*_\pi, \d', \d \sigma + i_N
\phi)$ is a quasi-Lie bialgebroid \cite{Roy}. Finally, a simple
computation shows that the double of this quasi-Lie bialgebroid is
naturally identified with the Courant algebroid  $(A \oplus
A^*)_{{\mathcal N}}^{\phi}$.
\end{proof}

\end{document}